\documentclass[12pt,amssymb, amscd, amstheorem,, amsmath, amstext]{amsart}

\usepackage{amsmath}
\usepackage{amssymb}
\usepackage{amsthm}
\usepackage{color}
\usepackage[all]{xy}
\usepackage{graphicx}

\newtheorem{theorem}{Theorem}
\newtheorem{definition}{Definition}
\newtheorem{lemma}{Lemma}
\newtheorem{corollary}{Corollary}
\newtheorem{prop}{Proposition}
\newtheorem{remark}{Remark}

\def\TE{\mathcal TE}

\def\E{\mathcal E}

\def\CC{\mathbb C}

\begin{document}

\title[Bounded Geometry and $(p,q)$-Exponential Maps]{Bounded Geometry and Characterization
of post-singularly Finite $(p,q)$-Exponential Maps}

\author{Tao Chen, Yunping Jiang, and Linda Keen}

\begin{abstract}
In this paper we define a topological class of branched covering maps of the plane called
{\em topological exponential maps of type $(p,q)$} and denoted by $\TE_{p,q}$, where $p\geq 0$ and $q\geq 1$.  
We follow the framework given in~\cite{Ji} to study the  problem of combinatorially characterizing an entire map 
$P e^{Q}$, where $P$ is a polynomial of degree $p$ and $Q$ is a polynomial of degree $q$ using an {\em iteration scheme defined by Thurston} and
a {\em bounded geometry condition}.  
We first show that  an element $f \in {\TE}_{p,q}$  with finite post-singular set
is combinatorially equivalent to an entire map  $P e^{Q}$ if and only if
it has bounded geometry with compactness. Thus to complete the characterization, we only need to check that the bounded geometry actually implies compactness. We show this for some 
$f\in \TE_{p,1}$, $p\geq 1$. Our main result in this paper is that a post-singularly finite map $f$ in $\TE_{0,1}$ or a post-singularly finite map $f$ in $\TE_{p,1}$, $p\geq 1$, with only one non-zero simple breanch point $c$ such that either $c$ is periodic or $c$ and $f(c)$ are both not periodic, is combinatorially equivalent to a post-singularly finite entire map of either the form $e^{\lambda z}$ or the form $ \alpha z^{p}e^{\lambda z}$, where $\alpha=(-\lambda/p)^{p}e^{- \lambda (-p/\lambda)^{p}}$, respectively,  if and only if it has  bounded geometry. This is the first result in this direction for a family of transcendental holomorphic maps with critical points.
\end{abstract}
\maketitle

\section{Introduction}
\label{sec:intro}

Thurston asked the question ``when can we realize a given branched covering map as a holomorphic map in such a way that the post-critical sets correspond?" and answered it
for post-critically finite degree $d$ branched covers of the sphere~\cite{T,DH}.
  His theorem is that a postcritically finite degree $d\geq 2$  branched covering of
the sphere, with hyperbolic orbifold, is either combinatorially equivalent
to a rational map or there is a  topological obstruction, now called a ``Thurston obstruction''.  
The proof uses an iteration scheme defined for an appropriate Teichm\"uller space.  The rational map, when it exists,  is unique up to conjugation by a
M\"obius transformation.

Although Thurston's iteration scheme is well defined for transcendental maps,   his theorem  does not naturally extend to them because the proof uses the finiteness of both the  degree and the post-critical set
in a crucial way.    
In this paper, we study a class of  entire maps: maps of the form $P e^{Q}$ where $P$ and $Q$ are polynomials of degrees $p\geq 0$ and $q\geq 0$ respectively.
This class, which we denote by $\E_{p,q}$  includes the exponential and polynomials.    We call the topological family analogue to these analytic maps $\TE_{p,q}$ and define it below.  Thurston's question makes sense for this family and our main theorem is an answer to this question.
Convergence of the iteration scheme depends on a compactness condition defined in section~\ref{sec:suff}.   We approach the question of compactness, using the idea of ``bounded geometry''.  This point of view was originally outlined in ~\cite{Ji} where the bounded geometry condition is an intermediate step that connects various topological  obstructions with the characterization of rational maps. The introduction of  this intermediate step makes  understanding the characterization of rational maps relatively easier and the arguments are more straightforward (see~\cite{Ji,JZ,CJ}). It also gives insight into the characterization of entire and meromorphic maps.  
 
In this paper we apply our techniques to characterize the  class of post-singularly finite entire maps with exactly one asymptotic value and finitely many critical points, the model topological space $\TE_{p,q}$.  
Our first result is 

\medskip
\begin{theorem}~\label{main1}
A post-singularly finite map $f$ in $\TE_{p,q}$ is combinatorially equivalent to a post-singularly finite entire map of the form $Pe^{Q}$  if and only if it has bounded geometry and satisfies the compactness condition.
\end{theorem}

We next prove that the compactness hypothesis holds  for a special family.  Our main result is

\medskip
\begin{theorem}[Main Theorem]~\label{main2}
A post-singularly finite map $f$ in $\TE_{0,1}$ or a post-singularly finite map $f$ in $\TE_{p,1}$, $p\geq 1$, with only one non-zero simple branch point $c$ such that either $c$ is periodic or $c$ and $f(c)$ are both not periodic, is combinatorially equivalent to a post-singularly finite entire map of either the form $e^{\lambda z}$ or the form $ \alpha z^{p}e^{\lambda z}$, where $\alpha=(-\lambda/p)^{p}e^{- \lambda (-p/\lambda)^{p}}$, respectively,  if and only if it has  bounded geometry.
\end{theorem}

\medskip
Our techniques involve adapting the Thurston iteration scheme to our situation.  We work with a fixed normalization.  There are two important parts to the proof of the main theorem (Theorem~\ref{main2}). The first is the proof of sufficiency under the assumptions of both bounded geometry and compactness. This part shows that bounded geometry together with the compactness assumption implies the convergence of the iteration scheme to an entire function of the same type (see section \ref{sec:suff}). Its proof involves  an analysis of  quadratic differentials associated to the functions in  the iteration scheme. The proof of the first part works for a more general post-singular finite map $f\in \TE_{p,q}$. The second part of the proof of the main theorem is in section~\ref{sec:proofmt} where we define a topological constraint. We prove in this part that the bounded geometry together with the topological constraint implies compactness. 
 
\begin{remark} 
Our main result is the first result to apply the Thurston iteration scheme to characterize  a family of transcendental holomorphic maps with critical points. Post-singularly finite maps $f$ in $\TE_{0,1}$ were also studied by Hubbard, Schleicher, and Shishikura~\cite{HSS} who used them to characterize a family of holomorphic maps with one asymptotic value, that is, the exponential family. In their study,  they used a different normalization; their functions take the form  $\lambda e^{z}$ so that they all have period $2\pi i$. They study the  limiting behavior of quadratic differentials associated to the exponential
 functions with finite post-singular set. They use a Levy cycle condition (a special type of Thurston's topological condition) to characterize when it is possible to realize a given exponential type map with finite post-singular set as an exponential map by combinatorial equivalence. Their calculations involve hyperbolic geometry.  
 
In this paper, by contrast, we normalize the  maps in the Thurston iteration scheme by assuming they all fix $0, 1, \infty$. We can therefore use spherical geometry in most our calculations rather than hyperbolic geometry. We are able to consider the more  general exponential maps  in  $\TE_{p,q}$.   The characterization theorem    for $f\in \TE_{p,1}$ in this paper,  is completely new. 
 \end{remark}  

\medskip
The paper is organized as follows. In \S2, we review the covering properties of $(p,q)$-exponential maps $E=Pe^{Q}$.
In \S3, we define the family $\TE_{p,q}$ of $(p,q)$-topological exponential maps $f$.
In \S4,
 we define the combinatorial equivalence between post-singularly
finite $(p,q)$-topological exponential maps and prove  there is a local quasiconformal $(p,q)$-topological exponential map in every combinatorial equivalence class of post-singularly finite $(p,q)$-topological exponential maps.
In \S5, we define the Teichm\"uller space $T_{f}$ for a post-singularly
finite $(p,q)$-topological exponential map $f$ and
in \S6, we introduce the induced map $\sigma_{f}$ from  the Teichm\"uller space $T_{f}$ into itself;  this is the crux of the Thurston iteration scheme.
In \S7, we define the concept of  ``bounded geometry''  and in \S8 we prove the necessity of the bounded geometry condition.
 In \S9, we give the proof of sufficiency assuming compactness.
The proofs we give in \S2-9 are for $(p,q)$-topological exponential maps $f$, $p\geq 0$ and $q\geq 0$.
 In \S10.1, \S10.2, and \S10.3, we define a topological constraint for the maps in our main theorem; this involves defining markings and a winding number.  
 We prove that the winding numbers are unchanged during iteration of the map $\sigma_f.$ Furthermore, in \S10.4 and \S10.5, we prove that the bounded geometry together with the topological constraint implies the compactness. This completes the proof of our main result.

In \S11, we make some  remarks about the relations between  ``bounded geometry''
and  ``canonical Thurston obstructions'' and between ``bounded geometry'' and  ``Levy cycles'' in the context of $\TE_{p,q}$.

 \medskip
{\bf Acknowledgement:} We would like to thank the referee whose careful reading and suggestions have much improved it.
The second and the third authors are partially supported by PSC-CUNY awards. The second author is also partially supported by the collaboration grant (\#199837) from the Simons Foundation, the CUNY collaborative incentive research grant (\#1861). This research is also partially supported by the collaboration grant (\#11171121) from the NSF of China
and a collaboration grant from Academy of Mathematics and Systems Science and the Morningside
Center of Mathematics at the Chinese Academy of Sciences.

\section{The space $\E_{p,q}$ of $(p,q)$-Exponential Maps}
\label{sec:epq}

\label{pq-exp maps}
We use the following notation:
  ${\mathbb C}$ is the complex plane, $\hat{\mathbb C}$ is the Riemann sphere and
  ${\mathbb C}^{*}$ is the complex plane  punctured at the origin.

A {\em $(p,q)$-exponential map} is an entire function of the form  $E =Pe^{Q}$ where $P$ and $Q$ are  polynomials of degrees $p\geq 0$ and $q\geq 0$ respectively such that  $p+q\geq 1$.  We use the notation ${\mathcal E}_{p,q}$ for the set of $(p,q)$-exponential maps.

Note that if $P(z)=a_0 + a_1 z + \ldots a_pz^p$, $Q(z)=b_0 + b_1 z + \ldots b_q z^q$, $\widehat{P}(z)=e^{b_0} P(z)$ and $\widehat{Q}(z)=Q(z)-b_0$ then
$$P(z) e^{Q(z)} = \widehat{P}e^{\widehat{Q}(z)}.$$
To avoid this ambiguity we always assume $b_0=0$.
If $q=0$, then $E$ is a polynomial of degree $p$.    Otherwise, $E$ is a transcendental entire function with essential singularity at infinity.

The growth rate of an entire function $f$ is defined as
$$
\limsup_{r\to \infty} \frac{\log \log M(r)}{\log r}
$$
where $M(r) =\sup_{|z|=r} |f(z)|$. It is easy to see that the growth rate of $E$ is $q$.

Recall that an asymptotic tract $V$ for an entire transcendental function $g$ is a simply connected unbounded domain  such that $g(V) \subset \hat\CC$ is conformally a punctured disk  $D \setminus \{a\}$ and the map $g:V \rightarrow g(V)$ is a universal cover.  The point $a$ is  the asymptotic value corresponding to the tract.    For functions $E$ in ${\mathcal E}_{p,q}$ we have

\begin{prop}
If $q\geq 1$, $E$ has $2q$ distinct asymptotic tracts that are separated by $2q$  rays.  Each tract maps to a punctured neighborhood of either zero or infinity and these are the only asymptotic values.
\end{prop}

\begin{proof}  From the growth rate of $E$ we see that for $|z|$ large, the behavior of the exponential dominates.   Since $Q(z) = b_q z^q + \mbox{ lower order terms} $,  in a neighborhood of infinity there are $2q$ branches of $\Re Q = 0$ asymptotic to equally spaced rays.  In the $2q$ sectors defined by these rays the signs of $\Re Q$ alternate.  If $\gamma(t)$ is a curve
such that $\lim_{t \to \infty} \gamma(t) = \infty$ and $\gamma(t)$ stays in one sector for all large $t$, then either $\lim_{t \to \infty}E(\gamma(t))=0$ or $\lim_{t \to \infty}E(\gamma(t))=\infty$, as $\Re Q$ is negative or positive in the sector.  It follows that there are exactly $q$ sectors that are asymptotic tracts for $0$ and $q$ sectors that are asymptotic tracts for infinity.    Because the complement of these tracts in a punctured neighborhood of infinity consists entirely of these rays, there can be no other asymptotic tracts.  \end{proof}

\begin{remark} The directions dividing the asymptotic tracts  are called {\em Julia rays} or {\em Julia directions} for $E$.  If $\gamma(t)$ tends to infinity along a Julia ray,  $E(\gamma(t))$ remains in a compact domain in the plane.   It spirals infinitely often around the origin.  \end{remark}

Two
$(p,q)$-exponential maps $E_{1}$ and $E_{2}$ are conformally equivalent if they are conjugate by a conformal automorphism $M$ of the Riemann sphere $\hat{\mathbb C}$, that is, $E_{1} =M\circ E_{2}\circ M^{-1}$. The automorphism $M$ must be a M\"obius transformation and it must fix both $0$ and $\infty$ so that it must be the affine stretch  map $M(z)=az$, $a \neq 0$. We are interested in conformal equivalence classes of maps, so by abuse of notation, we treat conformally equivalent $(p,q)$-exponential maps $E_{1}$ and $E_{2}$ as the same.

The critical points of $E=Pe^{Q}$ are the roots of $P'+PQ'=0$. Therefore, $E$ has $p+q-1$ critical points counted with multiplicity which we denote by
$$
\Omega_{E}=\{ c_{1}, \cdots, c_{p+q-1}\}.
$$
Note that if $E(z)=0$ then $P(z)=0$.  This in turn implies that if $c\in \Omega_E$ maps to $0$, then $c$ must also be a critical point of $P$.
Since $P$ has only $p-1$ critical points counted with multiplicity,
there must be at least $q$ points (counted with multiplicity) in $\Omega$ which are not mapped to $0$.
Denote by
$$
\Omega_{E,0}=\{ c_{1}, \cdots, c_{k}\}, \quad k\leq p-1,
$$
 the (possibly empty) subset of  $\Omega_E$ consisting of critical points such that $E(c_i)=0$.  Denote its complement in $\Omega_E$ by
$$
\Omega_{E,1} =\Omega_{E}\setminus \Omega_{E,0}=\{ c_{k+1}, \cdots, c_{p+q-1}\}.
$$

When $q=0$, $E$ is a polynomial. The {\em post-singular set} in this special case is  the same as the {\em post-critical set}. It is defined as
$$
P_{E} =\overline{\cup_{n\geq 1} E^{n}(\Omega_{E})}\cup\{\infty\}.
$$
To avoid trivial cases here we will assume that $\#(P_{E}) \geq 4$.  Conjugating by an affine map $z \to az+b$ of the complex plane,
	we normalize so that  $0, 1\in P_{E}$.

When $q=1$ and $p=0$, $\Omega_{E}=\emptyset$ and $\E_{0,1}$ consists of exponential maps $\alpha e^{\lambda z}$, $\alpha, \lambda \in \CC^*$. 
The {\em post-singular set} in this special case is defined as
$$
P_{E} =\overline{\cup_{n\geq 0} E^{n}(0)} \cup \{\infty\}.
$$
Conjugating by an affine stretch $z \mapsto \alpha z$ of the complex plane,
we normalize so that $E(0)=1$. Note that after this normalization the family takes the form
$e^{\lambda  z}$ , $\lambda \in \CC^{*}$.

When $q\geq 2$ and $p=0$ or when  $q\geq 1$ and $p\geq 1$, $\Omega_{E,1}$ is a non-empty set. Let
$$
{\mathcal V} =E(\Omega_{E,1}) =\{ v_{1}, \cdots, v_{m}\}
$$
denote the set of non-zero critical values of $E$.
The  {\em post-singular set}  for $E$  in  the general case  is now defined as

$$
P_{E} = \overline{\cup_{n\geq 0} E^{n}({\mathcal V}\cup \{ 0 \})} \cup \{\infty\}.
$$
We normalize  as follows:  \\

 If $E$ does not fix $0$, which is always true  if $q\geq 2$ and $p=0$,
we conjugate by an affine stretch $z \rightarrow az$ so that  $E(0)=1$.

If  $E(0)=0$,  there is a  critical point in  $c_{k+1}$ in $\Omega_{E,1}$  with $c_{k+1}\not=0$ and $v_{1}=E(c_{k+1})\not= 0$.
In this case we normalize so that $v_{1}=1$.  The family $\E_{1,1}$ consists of functions of the form  $\alpha z e^{\lambda z}$.
 After normalization  they  take  the form
$$
-\lambda e ze^{\lambda z}.
$$
An important family we consider in this paper is the family in $\E_{p,1}$, $p\geq 1$, where each map in this family 
has only one non-zero simple critical point. After normalization, the functions in this family take the form
$$
E(z)= \alpha z^{p} e^{\lambda z}, \quad \alpha=\Big( -\frac{\lambda}{p}\Big)^{p}  e^{- \lambda \Big(-\frac{p}{\lambda}\Big)^{p}}.
$$
This is the main family we will study in this paper. 

\section{Topological Exponential Maps of Type $(p,q)$}
\label{sec:Tpq}
We use the notation  ${\mathbb R}^{2}$ for the Euclidean plane.   We define the space ${\TE}_{p,q}$ of {\em topological exponential maps of type $(p,q)$}  with $p+q\geq 1$.    These are branched coverings with a single finite asymptotic value, normalized to be at zero, modeled on the maps in the holomorphic family $\E_{p,q}$.  For a full discussion of the covering properties for this family see Zakeri~\cite{Z}, and for a more general discussion of maps with finitely many asymptotic and critical values see Nevanlinna~\cite{Nev}.

If $q=0$, then ${\mathcal TE}_{p,0}$ consists of all topological polynomials $P$ of degree $p$:  these are degree $p$ branched coverings of the sphere such that $f^{-1}(\infty)=\{\infty\}$.

If $q=1$ and $p=0$, the space ${\TE}_{0,1}$ consists of universal covering maps $f: {\mathbb R}^{2}\to {\mathbb R}^{2}\setminus \{0\}$.  These are discussed at length in \cite{HSS}  where they are called topological exponential maps.

The polynomials $P$ and $Q$ contribute differently to the covering properties of maps in $\E_{p,q}$.  As we saw,  the degree of $Q$ controls the growth and behavior at infinity.    Using maps $e^{Q}$ as our model we first define  the space ${\TE}_{0,q}$.

\begin{definition}
If $q\geq 2$ and $p=0$, the space ${\TE}_{0,q}$ consists of topological branched covering maps $f: {\mathbb R}^{2}\to {\mathbb R}^{2}\setminus \{0\}$ satisfying the following conditions:
\begin{itemize}
\item[i)] The set of branch points,  $\Omega_{f} =\{c\in {\mathbb R}^{2}\;|\; \deg_{c}f\geq 2\}$  consists of $q-1$ points counted with multiplicity.
\item[ii)] Let ${\mathcal V} =\{ v_{1}, \cdots, v_{m}\} =f(\Omega_{f})\subset {\mathbb R}^{2}\setminus \{0\}$ be the set
           of distinct images of the branch points. For $i=1, \ldots, m$, let $L_{i}$  be a smooth topological ray in ${\mathbb R}^{2}\setminus \{0\}$ starting at  $v_{i}$  and extending to $\infty$ such that the collection of rays
            $\{L_{1}, \cdots, L_{m}\}$ are pairwise disjoint. Then
        \begin{enumerate}
              \item $f^{-1}(L_{i})$ consists of infinitely many  rays starting at points in the preimage set $f^{-1}(v_{i})$. If $x\in f^{-1}(v_{i})\cap \Omega_{f}$, there are $d_{x}=\deg_{x} f$ rays meeting at $x$ called {\em critical rays}.     If $x\in f^{-1}(v_{i})\setminus \Omega_{f}$, there is only one ray emanating from  $x$; it is called a {\em non-critical ray}.
              Set
              $$
              W=\mathbb R^2 \setminus  ( \cup_{i=1}^m L_i \cup \{0\}).
              $$
              \item The set of critical rays meeting at points in $\Omega_{f}$ divides $f^{-1}(W)$ into $q=1+\sum_{c\in \Omega_{f}} (d_{c}-1)$ open unbounded   connected components $W_{1}, \cdots, W_{q}$.
              \item[(3)] $f: W_{i}\to W $ is a universal covering for each $1\leq i\leq q$.
        \end{enumerate}
\end{itemize} \end{definition}

Note that the map restricted to each $W_i$ is a topological  model for the exponential map $z\mapsto e^{z}$ and the local degree at the critical points determines the number of $W_i$ attached at the point.

 We now define the space    ${\TE}_{p,q}$ in full generality where we assume $p>0$ and there is additional  behavior modeled on the role of  the new  critical points of $P e^{Q}$ introduced by  the non-constant polynomial $P$.

\medskip
\begin{definition}~\label{topexpdef}
If  $q\geq 1$ and $p\geq 1$, the space ${\TE}_{p,q}$ consists of topological branched covering maps $f: {\mathbb R}^{2}\to {\mathbb R}^{2}$ satisfying the following conditions:
\begin{itemize}
\item[i)] $f^{-1}(0)$ consists of $p$ points counted with multiplicity.
\item[ii)] The set of branch points, $\Omega_{f} =\{c\in {\mathbb R}^{2}\;|\; \deg_{c}f\geq 2\}$ consists of $p+q-1$  points counted with multiplicity.
\item[iii)] Let $\Omega_{f, 0} = \Omega_{f} \cap f^{-1}(0)$  be the $k<p$ branch points that map to $0$ and $\Omega_{f,1} =\Omega_{f}\setminus \Omega_{f,0}$  the $p+q-1-k$ branch points  that do not.   Note that $\Omega_{f,1}$ contains at least $q$ points and
  ${\mathcal V} =\{ v_{1}, \cdots, v_{m}\} =f(\Omega_{f,1})$  is contained in ${\mathbb R}^{2}\setminus \{0\}$. For $i=1, \ldots, m$, let $L_{i}$  be a smooth topological ray in ${\mathbb R}^{2}\setminus \{0\}$ starting at $v_{i}$ and extending to $\infty$ such that  the collection of rays  $\{L_{1}, \cdots, L_{m}\}$ are pairwise disjoint. Then
          \begin{enumerate}
                \item $f^{-1}(L_{i})$ consists of infinitely many  rays starting at points in the pre-image set $f^{-1}(v_{i})$.   If $x\in f^{-1}(v_{i}) \setminus \Omega_{f,1}$, there is only one ray emanating from $x$; this is a {\em non-critical ray}.  If $x\in f^{-1}(v_{i})\cap \Omega_{f, 1}$, there are $d_{x}=\deg_{x} f$  {\em critical rays} meeting at $x$.   Set $$W=\mathbb R^2 \setminus  ( \cup_{i=1}^m(L_i) \cup \{0\}).$$
                \item The collection of all critical rays meeting at points in $\Omega_{f,1}$ divides $f^{-1}(W)$ into $l=p+q-k=1+\sum_{c\in \Omega_{f,1}} (d_{c}-1)$  open unbounded connected components.
                \item Set $f^{-1}(0) =\{ a_{i}\}_{i=1}^{p-k}$  where the $a_i$ are distinct.   Each $a_i$ is contained in a distinct component of $f^{-1}(W)$;  label these components $W_{i,0}$, $i=1, \ldots p-k$.  Then the restriction   $f: W_{i,0} \setminus \{a_i\} \rightarrow W$ is an unbranched covering map  of degree $d_i=deg_{a_i}f$ where $d_i>1$ if $a_i \in \Omega_{f,0}$ and $d_i=1$ otherwise.
                  \item Label the remaining $q$ connected
                  components of $f^{-1}(W)$ by $W_{j,1}$, $j=1, \ldots, q$.
                  Then the restriction $f: {W_{j,1}} \rightarrow W$
                  is a universal covering map.
          \end{enumerate}
\end{itemize}
\end{definition}

%In Figure~1 we have drawn the configuration 
%for $f(z)= (z+1)e^{(1+4\pi i)z^2}$.  
%The larger points are the critical points 
%and the smaller points are zero, $1$ and 
%the critical values.  For computational reasons 
%the curves are the pre-images of full lines
% through the critical values and zero rather than rays.

%\begin{figure}[http]
%\begin{center}
%\includegraphics[width=6in]{"Inverse L Lines"}
%\caption{The lines $f^{-1}(L_i)$ for $f(z)= (z+1)e^{(1+4\pi i)z^2}$ with critical points and singular values}
%\end{center}
%\label{fig:1}
%\end{figure}

In section 3 of~\cite{ Z},  Zakeri proves that the $(p,q)$-exponential maps are topological exponential maps of type $(p,q)$.
The converse is also true.

\vspace*{5pt}
\begin{theorem}~\label{topexp}
Suppose $f\in {\TE}_{p,q}$ is analytic. Then $f=Pe^{Q}$ for two polynomials $P$ and $Q$ of degrees $p$ and $q$. That is, an analytic topological exponential map of type $(p,q)$ is a $(p,q)$-exponential map.
\end{theorem}

\begin{proof}
If $q=0$, then $f$ is a polynomial $P$ of degree $p$.

If $q\geq 1$, then $f$ is an entire function with $p$ roots, counted with multiplicity.
Every  such function can be expressed as
$$
f (z)= P(z) e^{g(z)}
$$
where $P$ is a polynomial of degree $p$ and $g$ is some entire function (see~\cite[Section 2.3]{Al}).

Consider  $$f'(z) = (P(z)g'(z)+P'(z))e^{g(z)}.$$
It is also an entire function,  and by assumption it has $p+q-1$ roots so that $Pg'+P'$ is a polynomial of degree $p+q-1$.  It follows that $g'$ is a polynomial of degree $q-1$ and $g=Q$ is a polynomial of degree $q$.
\end{proof}

Note that if $f \in  {\TE}_{p,q}$, $ q \neq 0$,  the origin plays a special role:  it is the only point with no or finitely many pre-images.  The conjugate of $f$ by $z \mapsto az$,  $ a \in {\mathbb C}^*$, is also in  ${\mathcal TE}_{p,q}$;  conjugate maps are conformally equivalent.

For $f\in {\TE}_{p,q}$, we define the {\em post-singular set} as follows:
\begin{itemize}
\item[i)] When $q=0$,  $E$ is a polynomial and, as mentioned in the introduction,  is treated elsewhere.   We therefore always assume $q \geq 1$.
\item[ii)] When $q=1$ and $p=0$, the {\em post-singular set} is
$$
P_{f} =\overline{\cup_{n\geq 0} f^{n}(0)}\cup\{\infty\}.
$$
We normalize so that $f(0)=1\in P_{f}$.
\item[iii)] When $q\geq 1$ and $p\geq 1$, the set of branch points is
$$\Omega_{f} =\{ c\in {\mathbb R}^{2}\;|\; \deg_{c} f \geq 2\}$$ and the {\em post-singular set} is
$$
P_{f} = \overline{\cup_{n\geq 0} f^{n}({\mathcal V}\cup \{ 0 \})} \cup \{\infty\}.
$$
If $q>1$ or if $q=1$ and $f(0) \neq 0$, we normalize so that $f(0)=1\in P_{f}$.
If $f(0)=0$, then, by the assumption $q \geq 1$, there is  a branch point $c_{k+1}\not=0$ such that $v_{1}=f(c_{k+1})\not= 0$.
We normalize so that $v_{1}=1$.
\end{itemize}
To avoid trivial cases we assume that $\#(P_{f}) \geq 4$.

It is clear that, in any case, $P_{f}$ is forward invariant, that is,
$$
f(P_{f}\setminus \{\infty\})\cup \{\infty\} \subseteq P_{f}
$$
or equivalently,
$$
f^{-1} (P_{f} \setminus \{\infty\})\cup \{\infty\} \supset P_{f}.
$$
Note that since we assume $q\geq 1$, $f^{-1}(P_{f}\setminus \{\infty\}) \setminus (P_{f}\setminus \{\infty\})$
contains infinitely many points.

\begin{definition}
We call $f\in {\TE}_{p,q}$  {\em post-singularly finite} if $\#(P_{f})<\infty$.
\end{definition}

\section{Combinatorial Equivalence}
\label{sec:combequiv}

\begin{definition}
Suppose $f, g$ are two post-singularly  finite maps in ${\mathcal TE}_{p,q}$. We say that they are {\em combinatorially equivalent}  if  they are topologically equivalent so that there are
homeomorphisms $\phi$ and $\psi$ of the sphere $S^{2}={\mathbb R}^{2}\cup\{\infty\}$ fixing $0$ and $\infty$ such that $\phi\circ f=g\circ \psi$ on ${\mathbb R}^{2}$ and if they satisfy the additional condition,  $\phi^{-1}\circ \psi$ is isotopic to the identity of $S^{2}$ rel $P_{f}$.
\end{definition}
The commutative diagram for the above definition is
\begin{equation*}
\xymatrix{\mathbb{R}^2 \ar[d]^f\ar[r]^{\psi} & \mathbb{R}^2\ar[d]^{g}\\
\mathbb{R}^2\ar[r]^\phi & \mathbb{R}^2}
\end{equation*}
The isotopy condition says that $P_{g}= \phi(P_{f})$.

Consider
${\mathbb R}^{2}\cup \{\infty\}$ equipped with the standard conformal structure as the Riemann sphere. Then $f\in {\mathcal T}E_{p,q}$ is a map from $\hat{\mathbb C}$ into itself.  We say $f\in {\TE}_{p,q}$ is {\em locally $K$-quasiconformal} for some $K>1$ if for any $z\in  \hat{\mathbb{C}} \setminus  \Omega_{f} \cup \{0\}$, there is a neighborhood $U$ of $z$ such that $f: U\to f(U)$ is $K$-quasiconformal.  Since we are working with isotopies rel a finite set, the following lemma is standard.

\begin{lemma}
Any post-singularly finite $f\in {\TE}_{p,q}$  is combinatorially equivalent to some locally $K$-quasiconformal map $g\in {\TE}_{p,q}$.
\end{lemma}

\begin{proof}
Recall that $\Omega_{f}$ is the set of branch points of $f$ in $\CC$ and $0$ is only asymptotic value in $\CC$.
Consider the space $X=\CC \setminus \Omega_{f}$. For every $p\in X$, let $U_{p}$ be a small neighborhood about $p$ such that $\phi_{p}=f|_{U}: U\to f(U)\subset \CC$ is injective. Then $\alpha=\{ (U_{p}, \phi_{p})\}_{p\in X}$ defines an atlas on $X$ with charts $(U_{p}, \phi_{p})$. If $U_{p}\cap U_{q}\not=\emptyset$, then $\phi_{p}\circ \phi_{q}^{-1} (z)=z: \phi_{q}(U_{p}\cap U_{q}) \to \phi_{p} (U_{p}\cap U_{q})$. Thus all transition maps are conformal ($1-1$ and analytic) and the atlas $\alpha$ defines a Riemann surface structure on $X$ which we again denote by $\alpha$.  Denote the Riemann surface by  $S=(X, \alpha)$.  From the uniformization theorem, $S$ is conformally equivalent to the Riemann surface $\CC\setminus A$ with the standard complex structure induced by  $\CC$,  where $A$ consists of $n=\#(\Omega_{f})+1$ points . The homeomorphism $h: \CC\to \CC$ with $h(0)=0$ and   $h: S=(X, \alpha)\to \CC\setminus A$ is conformal so that
$R=f\circ h^{-1}: \CC\to \CC$ is holomorphic with critical points at $h(\Omega_{f})$ and one asymptotic value at $0$. Since the set
$P_{f}$ is finite, following the standard procedure in quasiconformal mapping theory,  there is a $K$-quasiconfornal homeomorphism $k: \hat{\CC}\to \hat{\CC}$ such that $h$ is isotopic to $k$ rel $P_{f}$. The map  $g= R\circ k$ is a locally $K$-quasiconformal map in ${\TE}_{p,q}$ and combinatorially equivalent to $f$. This completes the proof of the lemma.
 \end{proof}

Thus without loss of generality, in the rest of the paper, we will assume that any post-singularly finite $f \in {\TE}_{p,q}$ is locally $K$-quasiconformal for some $K\ge 1$.

\section{Teichm\"uller Space $T_{f}$}
\label{sec:teich sp}

Recall that we denote $\mathbb{R}^2 \cup \{\infty\}$ equipped with the standard conformal structure by $\widehat{\mathbb{C}}$.  Let ${\mathcal M}=\{ \mu \in L^{\infty} (\widehat{\mathbb C})\;|\; \|\mu\|_{\infty}<1\}$ be the unit ball in the space of all measurable functions on the Riemann sphere. Each element $\mu\in {\mathcal M}$ is called a Beltrami coefficient.
For each Beltrami coefficient $\mu$, the Beltrami equation
$$
w_{\overline{z}}=\mu w_{z}
$$
has a unique quasiconformal  solution $w^{\mu}$ which maps $\hat{\mathbb C}$  to itself fixing $0,1, \infty$.
Moreover, $w^{\mu}$ depends holomorphically  on $\mu$.

Let $f$ be a post-singularly finite map in ${\mathcal TE}_{p,q}$  with post-singular set $P_f$.    The Teichm\"uller space $T(\hat{\mathbb C}, P_f)$ is defined as follows.  Given Beltrami differentials   $\mu, \nu \in {\mathcal M}$  we say that      $\mu$ and $\nu$ are equivalent in $\mathcal M$,  and denote this by  $\mu\sim \nu$, if $(w^{\nu})^{-1} \circ w^{\mu}$ is isotopic to the identity map of $\widehat{\mathbb C}$ rel $P_f$. The equivalence  class of $\mu$ under $\sim$ is denoted by $[\mu]$.
We set
$$
T_f=T(\hat{\mathbb C}, P_f)= \mathcal M/ \sim.
$$

It is easy to see that $T_{f}$ is a finite-dimensional complex manifold and is equivalent to the classical Teichm\"uller space $Teich(\hat{\mathbb C}\setminus P_{f})$ of Riemann surfaces with  basepoint $\hat{\mathbb C}\setminus P_{f}$. Therefore, the Teichm\"uller distance $d_{T}$ and the Kobayashi distance $d_{K}$ on $T_{f}$ coincide.

\section{Induced Holomorphic Map $\sigma_{f}$}
\label{sec:Thurston map}

For any post-singularly finite $f$ in ${\mathcal TE}_{p,q}$, there is an induced  map $\sigma= \sigma_{f}$ from $T_{f}$ into itself given by
$$
\sigma([\mu]) =[f^{*}\mu],
$$
where
\begin{equation}~\label{pullbackformula}
f^{*}\mu(z) = \frac{\mu_f(z) + \mu_f((f(z))
\theta(z)}{1 + \overline{\mu_f (z)} \mu_f(f(z)) \theta(z)}, \quad \theta(z) =\frac{\bar{f}_{z}}{f_{z}}.
\end{equation}
It is a holomorphic map so that

\vspace*{5pt}
\begin{lemma}~\label{contractive}
For any two points $\tau$ and $\tilde\tau$ in $T_{f}$,
$$
d_{T}(\sigma(\tau), \sigma(\tilde\tau))\leq d_{T}(\tau, \tilde\tau).
$$
\end{lemma}

 The next lemma follows directly from the definitions.
\vspace*{5pt}
\begin{lemma}
A  post-singularly finite $f$ in ${\TE}_{p, q}$ is combinatorially equivalent to a $(p,q)$-exponential map $E=Pe^{Q}$ iff $\sigma$ has a fixed point in $T_{f}$.
\end{lemma}

\section{Bounded Geometry}
\label{sec:bounded geometry}

For any $\tau_{0}\in T_{f}$, let $\tau_{n}=\sigma^{n}(\tau_{0})$, $n\geq 1$. The iteration sequence $\tau_{n}=[\mu_{n}]$ determines a sequence of finite subsets
$$
P_{f,n} = w^{\mu_{n}}(P_{f}), \quad n=0, 1, 2, \cdots.
$$
Since all $w^{\mu_{n}}$ fix $0, 1, \infty$, it follows that $0, 1, \infty\in P_{f,n}$.

\begin{definition}[Spherical Version]
We say $f$ has {\em bounded geometry} if there is a constant $b>0$ and a point $\tau_{0}\in T_{f}$ such that
$$
d_{sp} (p_{n},q_{n}) \geq b
$$
for $p_{n}, q_{n}\in  P_{f,n}$ and $n\geq 0$. Here
$$
d_{sp}(z,z')= \frac{|z-z'|}{\sqrt{1+|z|^{2}}\sqrt{1+|z'|^{2}}}
$$
is the spherical distance on $\hat{\mathbb C}$.
\end{definition}

Note that $d_{sp}(z, \infty) = \frac{|z|}{\sqrt{1+|z|^2}}$.
Away from infinity the spherical metric and Euclidean metric are equivalent.
Precisely, for any bounded  $S \subset \mathbb C$,
there is a constant $C>0$ which depends only on $S$ such that
$$
C^{-1} d_{sp}(x,y) \leq |x-y|\leq Cd_{sp}(x,y)\quad  \forall  x,y \in S.
$$

Consider the hyperbolic Riemann surface $R=\hat{\mathbb C}\setminus P_{f}$
equipped with the standard complex structure as the basepoint $\tau_{0}=[0]\in T_{f}$.
A point $\tau$ in $T_{f}$ defines another complex structure $\tau$ on $R$.
Denote by $R_{\tau}$  the hyperbolic Riemann surface  $R$ equipped with the complex structure $\tau$.

A simple closed curve $\gamma\subset R$ is called {\it non-peripheral} if each component of $\hat{\mathbb C}\setminus \gamma$ contains at least two points
of $P_{f}$. Let $\gamma$ be a non-peripheral simple closed curve in $R$.
For any $\tau=[\mu]\in T_{f}$, let $l_{\tau}(\gamma)$ be the hyperbolic length of the unique closed geodesic homotopic to $\gamma$ in $R_{\tau}$.
The bounded geometry property can be stated in terms of  hyperbolic geometry as follows.

\begin{definition}[Hyperbolic version]
We say $f$ has {\em bounded geometry}  if there is a constant $a>0$ and a point $\tau_{0}\in T_{f}$ such that
 $l_{\tau_{n}}(\gamma)\geq a$ for all $n\geq 0$ and all non-peripheral simple closed curves $\gamma$ in $R$.
\end{definition}

The above definitions of bounded geometry are equivalent because of the following lemma and the fact that we have normalized so that $0,1,\infty$ always belong to $P_f$.

\vspace*{5pt}
\begin{lemma}~\label{sg} Consider the hyperbolic Riemann surface
$\hat{\mathbb C}\setminus X$, where $X$ is a finite subset of $\hat\CC$ such that $0, 1, \infty \in X$, equipped with the standard complex structure.
Let $a>0$ be a constant. If every simple closed geodesic in $\hat{\mathbb C}
\setminus X$ has hyperbolic length greater than $a$, then the
spherical distance between any two distinct points in $X$ is bounded below by a bound $b>0$ which depends only on $a$ and $m=\#(X)$.
\end{lemma}

\begin{proof} If
$m=3$ there are no non-peripheral simple closed curves so in the following argument we may assume that $m\geq
4$. Let $X = \{x_{1}, \cdots, x_{m-1},x_{m} = \infty \}$   and let
$|\cdot|$ denote the Euclidean metric on ${\mathbb C}$.

Suppose $0=|x_{1}| \le \cdots \le |x_{m-1}|$. Let $M = |x_{m-1}|$.
Then $|x_{2}| \le 1$,  and we have
$$
\prod_{2 \le i \le m-2} \frac{|x_{i+1}|}{|x_{i}|} =
\frac{|x_{m-1}|}{|x_{2}|} \ge M.
$$
Hence
$$
\max_{2 \le i \le m-2} \Big\{ \frac{|x_{i+1}|}{|x_{i}|}\Big\} \ge
M^{\frac{1}{m-3}}.
$$
Let
$$
A_{i} = \{z\in {\mathbb C} \quad \Big{|}\quad |x_{i}| < z <
|x_{i+1}|\}
$$
and let $\mod(A_{i}) = \frac{1}{2\pi} \log\frac{|x_{i+1}|}{ |{x_i}|}$ be its modulus. Then for some
integer $2\leq i_{0}\leq m-2$  it follows that
$$
\mod(A_{i_{0}}) \ge \frac{\log M}{2 \pi (m-3)}.
$$
Denote the extremal length of  the core curve $\gamma_{i_0}$ in $A_{i_{0}} \subset \hat{\mathbb C}\setminus X$
by $\|\gamma_{i_{0}}\| $.
By properties of extremal length,
$$
\|\gamma_{i_{0}}\|= \frac{1}{\mod(A_{i_{0}})} \le \frac{2
\pi (m-3)}{\log M}.
$$
Since extremal length is defined by taking a supremum over all metrics and the area of $\hat{\mathbb C} \setminus X$  is $ 2\pi(m-2)$ for every hyperbolic metric,
$$
\|\gamma_{i_{0}}\|  \geq \frac{ l_{\tau_n}^{2}(\gamma)}{2\pi(m-2)} \geq  \frac{a^{2}}{2\pi(m-2)}.
$$
Setting $a'=  \frac{a^{2}}{2\pi(m-2)}$, these inequalities imply
$$
M \le M_{0}=e^{\frac{2 \pi (m-3)}{a'}}.
$$
Thus the spherical distance between $\infty$ and any finite point in
$X$ has a positive lower bound $b$ which depends only on $a$ and $m$.

Next we show that the spherical distance between any two
finite points in $X$ has a positive lower bound depending only on
$a$ and $m$.  By the equivalence of the spherical and Euclidean metrics in a bounded set in the plane,
it suffices to prove that $|x-y|$ is greater than a constant $b$ for any two finite points in $X$.

First consider the map $\alpha (z) =1/z$  which is a hyperbolic isometry from $X$ to $\alpha(X)$.  It preserves the set  $\{0, 1,\infty\}$ so that $0,1, \infty \in \alpha (X)$.
For any $2\leq i\leq m-1$, the above argument   implies that   $1/|x_{i}| \leq M_{0}$  and hence $|x_{i}|\geq 1/M_{0}$.
Similarly, for any $x_{i}\in X$,  $2\leq i\leq m-1$,  consider the map $\beta(z) =z/ (z-x_{i})$.  It maps $\{0,\infty,x_i\}$ to $\{0,1,\infty \}$ so that $\beta(X)$ contains $\{0,1,\infty\}$ and it is also a hyperbolic isometry.   For any $2\leq i\leq m-1$, the above argument implies that $|x_{j}|/|x_{j}-x_{i}| \leq M_{0}$
which in turn  implies that $|x_{j}-x_{i}| \geq 1/M_{0}^{2}$ proving the lemma.
\end{proof}

\section{The Main Result--Necessity}
\label{sec:main result}

Our main result (Theorem~\ref{main2}) has two parts: necessity and sufficiency. The necessity is relatively easy and can be proved for the general case. 
We prove the following statement. 

\vspace*{5pt}
\begin{theorem}\label{necc}
A post-singularly finite topological exponential map $f\in {\mathcal TE}_{p,q}$ is combinatorially equivalent to a unique $(p,q)$-exponential map
$E=Pe^{Q}$ if and only if
$f$ has bounded geometry.
\end{theorem}

\begin{proof}[Proof of Necessity]
If $f$ is combinatorially equivalent to
$E=Pe^{Q}$, then $\sigma$ has a unique fixed point $\tau_{0}$ so that $\tau_{n}=\tau_{0}$ for all $n$.  The complex structure on $\hat{\mathbb C} \setminus P_f$  defined by $\tau_0$ induces  a hyperbolic metric on it.   The shortest geodesic in this metric gives a lower bound on the lengths  of all geodesics  so that  $f$ satisfies the hyperbolic definition of bounded geometry.
\end{proof}

\section{Sufficiency under Compactness}
\label{sec:suff}

The proof of sufficiency of our main theorem (Theorem~\ref{main2}) is more complicated and needs some preparatory material.
There are two parts:  one is a compactness argument and the other is a fixed point argument.
Once one has compactness, the proof of the fixed point argument  is quite standard (see~\cite{Ji}) and works for any $f\in {\mathcal TE}_{p,q}$. 
This is the content of Theorem 1  which we prove in this section. 

The normalized functions in ${\mathcal E}_{p,q}$ are determined by the $p+q+1$ coefficients of the polynomials $P$ and $Q$.
This identification defines an embedding into $\CC^{p+q+1}$ and hence a topology on ${\mathcal E}_{p,q}$.

Given $f\in\TE_{p,q}$ and given any $\tau_{0}=[\mu_{0}]\in T_f$, let $\tau_{n}=\sigma^{n} (\tau_{0}) =[\mu_{n}]$
be the sequence generated by $\sigma$. Let $w^{\mu_{n}}$ be the normalized quasiconformal map with Beltrami coefficient $\mu_{n}$.
Then
$$
E_{n} = w^{\mu_{n}}\circ f\circ (w^{\mu_{n+1}})^{-1}\in {\mathcal E}_{p, q}
$$
since it preserves
$\mu_0$ and hence is holomorphic.
This gives a sequence $\{E_{n}\}_{n=0}^{\infty}$ of maps in ${\mathcal E}_{p,q}$ and a sequence of subsets
$P_{f,n}=w^{\mu_{n}} (P_{f})$.
Note that $P_{f,n}$ is not, in general, the post-singular set $P_{E_{n}}$ of $E_{n}$.
If it were, we would have a fixed point of $\sigma$.

\medskip
\noindent {\bf The compactness condition.}  We say $f$ satisfies the compactness condition if the sequence $\{ E_n \}_{n=1}^{\infty}$ generated  in the Thurston iteration scheme is contained in a compact subset of $ \E_{p,q}$.

\medskip
From a conceptual point of view, the compactness condition is very natural and simple. From a technical point of view, however, it is not at all obvious.  We give a detailed proof showing how to get a fixed point from the condition of bounded geometry with compactness. 

Suppose $f$ is a post-singularly finite topological exponential map in ${\mathcal TE}_{p,q}$.
For any $\tau=[\mu]\in T_{f}$, let $T_{\tau}T_{f}$ and $T^{*}_{\tau}T_{f}$ be the tangent space and the cotangent space of $T_{f}$ at $\tau$ respectively.  Let $w^{\mu}$ be the corresponding normalized quasiconformal map fixing $0,1, \infty$.
Then $T^{*}_{\tau}T_{f}$ coincides with the space ${\mathcal Q}_{\mu}$ of integrable meromorphic quadratic differentials $q=\phi(z) dz^{2}$.  Integrablility means that  the norm of $q$,  defined by
$$
||q|| =\int_{\hat{\mathbb C}} |\phi(z)| dzd\overline{z}
$$
is finite.  This condition implies that  the  poles of $q$  must occur at points of $w^{\mu} (P_{f})$ and that these poles are simple.

Set $\tilde{\tau}=\sigma(\tau)=[\tilde{\mu}]$ and denote by $w^{\mu}$ and $w^{\tilde{\mu}}$   the corresponding  normalized quasiconformal maps. We have the following commutative
diagram:
$$
\begin{array}{ccc} \hat{\mathbb C}\setminus f^{-1}(P_{f}) & {\buildrel w^{\tilde{{\mu}}} \over
\longrightarrow} & \hat{\mathbb C}\setminus w^{\tilde{{\mu}}} (f^{-1}(P_{f}))\cr \downarrow f
&&\downarrow E_{\mu,\tilde{\mu}}\cr \hat{\mathbb C}\setminus P_{f}& {\buildrel w^{\mu}
\over \longrightarrow} & \hat{\mathbb C}\setminus w^{\mu}(P_{f}).
\end{array}
$$
Note that in the diagram, by abuse of notation, we write  $f^{-1}(P_{f})$ for  $f^{-1}(P_{f}\setminus \{\infty\})\cup\{\infty\}$.  Since by definition  $\tilde{\mu} = f^*\mu$, the map
  $E=E_{\mu,\tilde{\mu}}=  w^{\mu} \circ f \circ  (w^{\tilde{\mu}})^{-1}$ defined on ${\mathbb C}$
is analytic. By Theorem~\ref{topexp}, $E_{\mu,\tilde{\mu}}=P_{\tau, \tilde{\tau}}e^{Q_{\tau,\tilde{\tau}}}$ for a pair of polynomials $P=P_{\tau, \tilde{\tau}}$ and $Q=Q_{\tau,\tilde{\tau}}$ of respective degrees $p$ and $q$.

Let $\sigma_{*}: T_{\tau}T_{f}\to T_{\tilde{\tau}}T_{f}$ and  $\sigma^{*}: T_{\tilde{\tau}}^*T_{f}\to T_{\tau}^*T_{f}$ be the tangent and co-tangent map of $\sigma$, respectively.

Let $\beta(t)=[\mu(t)]$ be a smooth path in $T_{f}$ passing though $\tau$ at $t=0$ and let $\eta=\beta'(0)$ be the corresponding tangent vector at $\tau$.   Then the
pull-back $\tilde{\beta}(t)= [f^{*}\mu (t)]$ is a smooth path in $T_{f}$ passing though
$\tilde{\tau}$ at $t=0$ and $\tilde{\eta} = \sigma_{*} \eta=\tilde{\beta}'(t)$ is the corresponding tangent vector  at $\tilde{\tau}$.   We move these tangent vectors  to the origin in $T_f$ obtaining the vectors $\xi, \tilde\xi$ using the maps
$$
\eta = (w^{\mu})^{*}\xi \quad \hbox{and}\quad \tilde{\eta} =
(w^{\tilde{\mu}})^{*} \tilde{\xi}.
$$
This gives us the following commutative diagram:
$$
\begin{array}{ccc} \tilde{\eta}& {\buildrel (w^{\tilde{\mu}})^{*} \over
\longleftarrow} & \tilde{\xi}\cr \hbox{\hskip17pt}\uparrow f^{*} &
& \hbox{\hskip17pt}\uparrow E^{*} \cr \eta & {\buildrel (w^{\mu})^{*} \over \longleftarrow}& \xi\cr
\end{array}
$$
Now suppose $\tilde{q}$ is a co-tangent vector in
$T^{*}_{\tilde{\tau}}$ and let
$q=\sigma^{*} \tilde{q}$ be the corresponding co-tangent vector in
$T^{*}_{\tau}$. Then
$\tilde{q}=\tilde{\phi} (w) dw^{2}$ is an integrable quadratic differential on $\hat{\mathbb C}$ that can have at worst simple poles along $w^{\tilde{\mu}}(P_{f})$
and $q=\phi (z) dz^{2}$ is an integrable quadratic differential on $\hat{\mathbb C}$ that can have at worst simple poles along
$w^{\mu}(P_{f})$. This implies that $q=\sigma_{*}\tilde{q}$
is also the push-forward integrable quadratic differential
$$
q=E_{*}\tilde{q} =\phi (z) dz^{2}
$$
of $\tilde{q}$ by $E$.   To see this, recall from section~\ref{sec:Tpq} that $E$, and a choice of curves $L_i$ from the branch points, determine a finite set of domains $W_i$ on which $E$ is an unbranched covering to a domain homeomorphic to $\CC^*$.  Since $E$ restricted to each $W_i$ is either a topological model for $e^z$ or $z^k$, we may divide each $W_i$ into a collection of  fundamental domains on which $E$ is bijective. Therefore the coefficient $\phi (z)$ of $q$ is given by the formula
\begin{equation}~\label{pushforwardformula}
\phi(z)= ({\mathcal L}\tilde{\phi}) (z) =\sum_{E(w) = z}
\frac{\tilde{\phi} (w)}{(E'(w))^{2}} = \frac{1}{z^{2}} \sum_{E(w)=z} \frac{\tilde{\phi} (w)}{(\frac{P'(w)}{P(w)} +Q'(w))^{2}}
\end{equation}
where ${\mathcal L}$ is the standard transfer operator and $\tilde{\phi}$ is the coefficient of $\tilde{q}$. Thus
\begin{equation}~\label{pushforwardformula1}
q= \phi(z) dz^{2}= \frac{dz^{2}}{z^{2}} \sum_{E(w)=z} \frac{\tilde{\phi} (w)}{(\frac{P'(w)}{P(w)} +Q'(w))^{2}}
\end{equation}

It is clear that as a  quadratic differential defined on $\hat{\mathbb C}$,   we have
$$
||q||\leq ||\tilde{q}||.
$$
Since $q$ is integrable and $0$ and $\infty$ are isolated singularities, it follows that $q$ has at worst possible simple poles at these points so that the inequality holds on all of $\hat{\mathbb C}$.

By  formula~(\ref{pushforwardformula}), we have
$$
\langle \tilde{q} ,\tilde{\xi} \rangle =  \langle q, \xi \rangle
$$
which implies
$$
\|\tilde{\xi}\|  \leq \|\xi\|
$$
where this is the $L^{\infty}$ norm.
This gives another proof of Lemma~\ref{contractive}.
Furthermore, we have the following  stronger assertion

\medskip
\begin{lemma}~\label{infstrongcon}
$$
||q||<||\tilde{q}||
$$
and
$$
\|\tilde{\xi}\| < \|\xi\|.
$$
\end{lemma}

\begin{proof}
Suppose there is a $\tilde{q}$ such that $||q||=||\tilde{q}||\not= 0$.
Using the change of variable $E(w)=z$ on each fundamental domain we get
$$
\int_{\hat{\mathbb C}} \Big| \sum_{E(w)=z} \frac{\tilde{\phi} (w)}{(E'(w))^{2}}\Big| dz  \, d\overline{z}=
\int_{\hat{\mathbb C}} |\phi(z)|dz  \, d\overline{z} =\int_{\hat{\mathbb C}} |\tilde{\phi} (w)| \, dw d\overline{w}
$$
$$
= \sum_{i} \int_{W_{i}} |\tilde{\phi} (w)| \, dw d\overline{w}  = \int_{\hat{\mathbb C}} \sum_{i} \Big|\frac{\tilde{\phi} (w)}{(E'(w))^{2}}\Big| \, dz d\overline{z}.
$$
By the triangle inequality,   all the  factors $\frac{\tilde{\phi} (w)}{(E'(w))^{2}}$ in  $\sum_{E(w)=z} \frac{\tilde{\phi} (w)}{(E'(w))^{2}}$ have the same argument. That is,  there is a real number $a_{z}$ for every $z$ such that for any pair of points $w,w'$ with $E(w)=E(w')=z$,
$$
\frac{\tilde{\phi} (w)}{(E'(w))^{2}}=a_{z} \frac{\tilde{\phi} (w')}{(E'(w'))^{2}}.
$$
 Now formula~(\ref{pushforwardformula})   implies
$\phi (z) =\infty$ which cannot be;  this contradiction proves the lemma.
\end{proof}

\medskip
\begin{remark}
The real point here is that $E$ has infinite degree and any   $q$ has finitely many poles.  If  there were a $\tilde{q}$ with  $||q||=||\tilde{q}||\not= 0$ and if $Z$ is the set of poles of $\tilde{q}$, then
the poles of $q$ would be contained in the set $E(Z)\cup {\mathcal V}_{E}$, where ${\mathcal V}_{E}$ is the set of critical values of $E$.
Thus, by formula~(\ref{pushforwardformula}),
$$
E^{*} q =\phi (E(w)) dw^2 = d \tilde{q} (w),
$$
where $d$ is the degree of $E$. Furthermore,
$$
E^{-1} (E(Z)\cup {\mathcal{V}}_{E}) \subseteq Z \cup {\Omega}_{E}.
$$
Since  $d$ is infinite,  the last inclusion formula can not hold  since the left hand side is infinite and the right hand side is finite.
\end{remark}

 An immediate corollary is
 
\medskip 
\begin{corollary}~\label{strongcon}
For any two points $\tau$ and $\tilde{\tau}$ in $T_{f}$,
$$
d_{T}\Big(\sigma(\tau), \sigma(\tilde{\tau})\Big)< d_{T}(\tau, \tilde{\tau}).
$$
\end{corollary}

Furthermore,

\medskip
\begin{lemma}~\label{fixedpt}
If $\sigma$ has a fixed point in $T_{f}$, then this fixed point must be unique. This is equivalent to saying  that
 a post-singularly finite $f$ in ${\TE}_{p, q}$  is combinatorially equivalent to at most one  $(p,q)$-exponential map $E=Pe^{Q}$.
\end{lemma}

We can now finish the proof of the sufficiency in Theorem~\ref{main1}.

\begin{proof}[Proof of Theorem~\ref{main1}] 
Suppose $f\in {\TE}_{p,q}$ has bounded geometry with compactness.
Recall that the map defined by
\begin{equation}\label{thu iter}
E_{n}=w^{\mu_{n}}\circ f\circ (w^{\mu_{n+1}})^{-1}
\end{equation}
is a $(p,q)$-exponential map.

If $q=0$, $E_n$ is a polynomial and the theorem follows from the arguments given in~\cite{CJ} and \cite{DH}.
Note that if $P_f = \{0,1,\infty\}$, then $f$ is a universal covering map of $\mathbb{C}^*$ and is therefore combinatorially equivalent to $e^{2 \pi i z}$.
Thus in the following argument, we assume that $\#(P_f) \geq 4$. Then, given our normalization conventions and the bounded geometry hypothesis we see that
the functions  $E_{n}$, $n=0, 1, \ldots$ satisfy the following conditions:
\begin{itemize}
\item[1)] $m=\#(w^{\mu_{n}}(P_f))\geq 4$ is fixed.
\item[2)] $0, 1, \infty \in w^{\mu_{n}}(P_f)$.
\item[3)] $\Omega_{E_{n}}\cup \{ 0, 1, \infty\} \subseteq
E_{n}^{-1}(w^{\mu_{n}}(P_f))$.
\item[4)] there is a $b>0$ such that $d_{sp}(p_{n}, q_{n}) \geq b$ for any $p_{n}, q_{n}\in w^{\mu_{n}}(P_f)$.
\end{itemize}

As a consequence of the compactness, we have that in the sequence $\{ E_{n}\}_{n=1}^{\infty}$, there is a subsequence $\{ E_{n_{i}}\}_{i=1}^{\infty}$ converging to a map $E=Pe^{Q} \in \E_{p,q}$ where $P$ and $Q$ are polynomials of degrees $p$ and $q$ respectively.

Any integrable quadratic differential $q_{n} \in T^{*}_{\tau_{n}}{\mathcal T}_{f}$ has, at worst, simple poles in the finite set  $P_{n+1,f}= w^{\mu_{n+1}}(P_f)$.
Since $T^{*}_{\tau_{n}}{\mathcal T}_{f}$ is a finite dimensional linear space, there is a quadratic differential $q_{n, max}\in T^{*}_{\tau_{n}}{\mathcal T}_{f}$
with $\| q_{n,max}\|=1$ such that
$$
0 \leq a_{n}=\sup_{||q_{n}||=1} \|(E_{n})_{*}q_{n}|| = \|(E_{n})_{*}q_{n,max}\| <1.
$$
Moreover, by the bounded geometry condition,  the possible simple poles of $\{q_{n, max}\}_{n=1}^{\infty}$ lie in a
  compact set and hence these quadratic differentials lie in a compact subset of the  space of quadratic differentials on $\hat{\mathbb C}$ with, at worst, simples poles at $m=\#(P_{f})$ points.

Let
$$
a_{\tau_{0}} =\sup_{n\geq 0} a_{n}.
$$
Let  $\{n_{i}\}$  be a sequence of  integers such that the subsequence $a_{n_{i}}\to a_{\tau_{0}}$ as $i\to \infty$.
By compactness, $\{E_{n_{i}}\}_{i=0}^{\infty}$ has
a convergent subsequence, (for which we  use the same notation)
that converges to a holomorphic  map $E \in \E_{p,q}$.
Taking a further subsequence if necessary, we obtain a convergent sequence of sets $P_{n_{i},\tau_{0}} =w^{\mu_{n_{i}}}(P_{f})$  with limit set $X$.
By bounded geometry, $\#(X)=\#(P_{f})$ and  $d_{sp}(x, y) \geq b$ for any $x,y\in X$.
Thus we can find a   subsequence $\{ q_{n_{i}, max}\}$ converging to an integrable quadratic differential $q$ of norm $1$  whose only poles lie in $X$ and are simple.
Now by lemma~\ref{infstrongcon},
$$
a_{\tau_{0}} = ||E_{*} q|| <1.
$$

Thus we have proved  that there is an  $0< a_{\tau_{0}}< 1$, depending only on $b$ and $f$, such that
$$
\|\sigma_{*}\| \le
\|\sigma^{*}\| \le a_{\tau_{0}}.
$$
Let $l_{0}$ be a curve connecting $\tau_{0}$ and $\tau_{1}$ in $T_{f}$ and set $l_{n}=\sigma_{f}^{n}(l_{0})$ for $n\geq 1$. Then $l=\cup_{n=0}^{\infty}l_{n}$ is a curve in $T_f$
 connecting all the points $\{\tau_{n}\}_{n=0}^{\infty}$. For each point $\tilde{\tau}_{0}\in l_{0}$, we have $a_{\tilde{\tau}_{0}} <1$. Taking the maximum  gives  a uniform $a<1$ for all points in $l_0$.  Since $\sigma$ is holomorphic, $a$ is an upper bound for all points in $l$.  Therefore,

$$
d_{T} (\tau_{n+1}, \tau_{n}) \leq a \, d_{T}(\tau_{n}, \tau_{n-1})
$$
for all $n\geq 1$.
Hence, $\{ \tau_{n}\}_{n=0}^{\infty}$ is a convergent sequence with a unique limit point $\tau_{\infty}$ in $T_{f}$ and $\tau_{\infty}$  is
a fixed point of $\sigma$. This combining with Lemma~\ref{fixedpt} completes the proof of the sufficiency of Theorem~\ref{main1}.
\end{proof}

From our proofs of Theorem~\ref{necc} and Theorem~\ref{main1},  the final step in the   proof of the main theorem is to prove the compactness condition holds from bounded geometry.  This is in contrast to  the case of rational maps (see~\cite{Ji}) where the bounded geometry condition guarantees the compactness condition holds. In the case of $(p,q)$-exponential maps, 
the bounded geometry condition must be combined with some topological constraints to guarantee the compactness. The  topological constraints, together with the bounded geometry condition control the sizes of fundamental domains so that they are neither too small nor too big. Thus, before we prove      the compactness condition holds, we will describe a topological constraint for the two types of  map $f$ in the Main Theorem (Theorem~\ref{main2}).

\section{Proof of the Main Theorem (Theorem~\ref{main2}).}
\label{sec:proofmt}
In section~\ref{sec:Tpq} we defined two different  normalizations   for functions in $\TE_{p,q}$ that depend on whether or not $0$ is a fixed point of the map.  The topological constraints 
for post-singularly finite maps also follow this dichotomy.    

\medskip
\subsection{A topological constraint for $f\in \TE_{0,1}$ satisfying the hypotheses of Theorem 2.}
\label{top1}
Any such $f$ has no branch points  so $P_{f} =\cup_{k\geq 0} f^{k} (0) \cup\{\infty\}$ which is finite.  Since $0$ is omitted, the orbit of $0$ is pre-periodic. Let $c_{k}=f^{k}(0)$ for $k\geq 0$. By the pre-periodicity, there are  integers $k_{1}\geq 0$ and $l\geq 1$ such that 
$f^{l}(c_{k_{1}+1}) =   c_{k_{1}+1}$. This says that 
$$
\{ c_{k_{1}+1}, \ldots, c_{k_{1}+l}\}
$$
is a periodic orbit of period $l$. Let $k_{2} =k_{1}+l$.  
  Let $\gamma$ be a continuous curve 
connecting $c_{k_{1}}$ and $c_{k_{2}}$ in 
${\mathbb R}^{2}$ disjoint from $P_f$,  except for its endpoints.  Because 
$$
f(c_{k_{1}})=f(c_{k_{2}})=c_{k_{1}+1},
$$ 
the image curve $\delta= f(\gamma)$  is a closed curve.  

\medskip
\subsection{A topological constraint for $f\in \TE_{p,1}$ satisfying the hypotheses of Theorem 2.}
\label{top2}
Any such $f$ has exactly one non-zero simple branch point which we denote by $c$; 
$0$ is the only other branch point and it  has multiplicity $p-1$.  Then $f(0)=0$ and by our normalization, $f(c)=1$.  In this case
$$
P_{f} =\cup_{k\geq 1} f^{k} (c) \cup\{0, \infty\}.
$$
Again by the hypothesis of Theorem~\ref{main2},  $P_{f}$ is finite. Set $c_{k}=f^{k}(c)$ for $k\geq 0$. 

Suppose $c$ is not periodic. As above, there are integers $k_{1}\geq 0$ and $l\geq 1$ such that 
$f^{l}(c_{k_{1}+1}) =   c_{k_{1}+1}$.  Again,  
$$
\{ c_{k_{1}+1}, \ldots, c_{k_{1}+l}\}
$$
is a periodic orbit of period $l$. Let $k_{2} =k_{1}+l$.  

As above, let $\gamma$ be a continuous curve 
connecting $c_{k_{1}}$ and $c_{k_{2}}$ in 
${\mathbb R}^{2}$ disjoint from $P_f$,  except for its endpoints. Since 
$$
f(c_{k_{1}})=f(c_{k_{2}})=c_{k_{1}+1},
$$ 
the image curve $\delta= f(\gamma)$  is a closed curve.  

\subsection{Winding numbers}
\label{winding}
In each of the above cases, the {\em winding number} $\eta$ of the closed curve $\delta= f(\gamma)$ about $0$ essentially counts the number of fundamental domains 
between $c_{k_{1}}$ and $c_{k_{2}}$  and defines the ``distance'' between the fundamental domains.    
The following lemma is a crucial to proving that the compactness condition holds for each type of function in Theorem~\ref{main2}.

 \medskip 
\begin{lemma}\label{winding1}
The winding number $\eta$ is  does not change  under the Thurston iteration procedure.
\end{lemma}
\begin{proof}

Given $\tau_{0}=[\mu_{0}]\in T_f$, let $\tau_{n}=\sigma^{n} (\tau_{0}) =[\mu_{n}]$
be the sequence generated by $\sigma$. Let $w^{\mu_{n}}$ be the normalized quasiconformal map with Beltrami coefficient $\mu_{n}$.
Then in either of the above situations, 
$$
E_{n} = w^{\mu_{n}}\circ f\circ (w^{\mu_{n+1}})^{-1}\in {\mathcal E}_{p, 1} 
$$
since it  preserves $\mu_0$ and  is holomorphic. See the following diagram.
$$
\begin{array}{ccc} \hat{\mathbb C}& {\buildrel w^{\mu_{n+1}} \over
\longrightarrow} & \hat{\mathbb C}\cr
\downarrow f &&\downarrow E_{n}\cr
\hat{\mathbb C}& {\buildrel w^{\mu_{n}}
\over \longrightarrow} & \hat{\mathbb C}.
\end{array}
$$
Let $c_{k, n} =w^{\mu_{n}}(c_{k})$. 
The continuous curve
$$
\gamma_{n+1}=w^{\mu_{n+1}} (\gamma)
$$
goes from  $c_{k_{1}, n+1}$ to $c_{k_{2}, n+1}$.  
The image curve is 
$$
\delta_{n}=E_{n} (\gamma_{n+1}) =   w^{\mu_{n}}(f((w^{\mu_{n+1}})^{-1} (\gamma_{n+1})))=w^{\mu_{n}} (f (\gamma)) = w^{\mu_{n}} (\delta).
$$
Note that $w^{\mu_{n}}$ fixes $0,1,\infty$. Thus $\delta_{n}$ is a closed curve  through the point $c_{k_{1}+1, n}= w^{\mu_{n}} (c_{k_{1}+1})$ and it has    winding number $\eta$ around $0$.
\end{proof}
 
 The argument that this invariance  plus bounded geometry implies the compactness is different in each of these two cases. We present these arguments in the two subsections below.  Let set 
$$
P_{n}=P_{f, n} =w^{\mu_{n}} (P_{f}), \quad n=0, 1, 2, \ldots.
$$

 \subsection{The compactness condition  for $f\in \TE_{0,1}$ satisfying the hypotheses of  Theorem~\ref{main2}.}
 \label{comp1} 
In this case, all the functions in the Thurston iteration have the form $E_{n} (z) = e^{\lambda_{n} z}$.   From our normalization, we have that 
$$
0, \; 1=E_{n}(0),\;  E_{n}(1) =e^{\lambda_{n}} \in P_{n+1}.
$$  
Note that in this case $E_{n}(1) \neq1$.  When $f$ has bounded geometry,  
the spherical distance between $1$ and  $E_{n}(1)$ is bounded away from zero.  
That is,  there is a constant $\kappa>0$ such that
$$
\kappa \leq |\lambda_{n}|, \;\; \forall n\geq 0.
$$
 
Now we prove that the sequence $\{|\lambda_{n}|\}$ is also bounded above. 
We can compute
$$
\eta = \frac{1}{2\pi i} \oint_{\delta_{n}} \frac{1}{w} dw = \frac{1}{2\pi i} \int_{\gamma_{n+1}} \frac{E_{n}'(z)}{E_{n}(z)} dz= \frac{1}{2\pi i} \int_{\gamma_{n+1}} \lambda_{n}dz.
$$
The integral therefore depends only on the endpoints and  we have 
$$
\eta= \frac{1}{2\pi i} \int_{\gamma_{n+1}} \lambda_{n}dz=\frac{\lambda_{n}}{2\pi i} (c_{k_{2}, n+1}-c_{k_{1}, n+1}). 
$$  
Since $0$ is omitted, it can not be periodic.  Therefore, both $c_{k_{2}, n+1}\not= c_{k_{1}, n+1}\in P_{n+1}$,  
so by bounded geometry  there is a positive constant which we again  denote by $\kappa$ such that 
$$
|c_{k_{2}, n+1}-c_{k_{1}, n+1}|\geq \kappa. 
$$ 
This gives us the estimate 
$$
|\lambda_{n}| \leq \frac{2\pi \eta}{|c_{k_{2}, n+1}-c_{k_{1}, n+1}|}\leq \frac{2\pi \eta}{\kappa}. 
$$
which proves that $\{ E_{n}(z)\}_{n=0}^{\infty}$ is contained in a compact family in $\E_{0,1}$. This combined with Theorem~\ref{necc} and Theorem~\ref{main1} completes the proof of Theorem~\ref{main2} for $f\in \TE_{0,1}$.   

\subsection{The compactness condition for $f\in TE_{p,1}$ satisfying the hypotheses in Theorem~\ref{main2}.}
\label{comp2} 
For such a map,  $f(0)=0$ and $0$ is a branch point of multiplicity $p-1$. And $f$ has exactly one non-zero branch point $c$ with $f(c)=1$. All the functions in the Thurston iteration have forms
$$
E_{n} (z) = \alpha_{n} z^{p} e^{\lambda_{n} z}, \quad
\alpha_{n}=e^{p} \Big( -\frac{\lambda_{n}}{p}\Big)^{p}.
$$
Note that $E_{n}(0)=0$ and $0$ is a critical point of multiplicity $p-1$. It is also the asymptotic value and hence it has no other pre-images.  
Moreover, $E_{n}(z)$ has exactly one non-zero simple critical point 
$$
c_{n}=-\frac{p}{\lambda_{n}}=w^{\mu_{n}}(c).
$$
and $\alpha_{n}$ is defined by the normalization condition $E_{n}(c_{n})=1$.
    
If $c$ is periodic, then $c\in P_{f}$. This implies that $c_{n} \not= 0, \infty\in P_{n}$ and thus its spherical distance from either $0$ or $\infty$ is bounded below.  That is,  there are two constants $0<\kappa<K<\infty$ such that
$$
\kappa\leq |\lambda_{n}|\leq K, \quad \;\forall n>0.
$$   
This implies that the sequence $\{E_{n}\}_{n=1}^{\infty}$ is contained in a compact subset. 

Now suppose $c$ is not periodic. By the the hypotheses in Theorem~\ref{main2}, $f(c)=1$ is also not periodic. 
This implies that $k_{1}\geq 1$. 

Recalling the notation $P_{n} =w^{\mu_{n}} (P_{f})$,  we have 
$$
0, \quad 1=E_{n}(c_{n}), \quad E_{n}(1) = e^{p} \Big( -\frac{\lambda_{n}}{p}\Big)^{p} e^{\lambda_{n}} \in P_{n}.
$$
Let $c_{k, n} =w^{\mu_{n}}(c_{k})$. Then $c_{k,n}\in P_{n}$ for all $k\geq 1$.  
Let $\delta_{n} =w^{\mu_{n}}(\delta)$ and $\gamma_{n} =w^{\mu_{n}}(\gamma)$.  

When $f$ has bounded geometry,  since $E_{n}(1) \not=0$, its spherical   distance from $0$ is bounded below.   This implies that  the sequence $\{|\lambda_{n}|\}$ is bounded below; that is, there is a constant $\kappa>0$ such that
$$
\kappa\leq |\lambda_{n}|, \;\; \forall n>0.
$$

By our hypothesis,  $c_{k_{1},n+1}\not= c_{k_{2},n+1}$ both belong to $P_{n+1}$ and bounded geometry  implies there are two constants, which we still denote by
$0<\kappa< K<\infty$, such  that 
$$
\kappa\leq |c_{k_{2}, n+1}|, \;\;  |c_{k_{1}, n+1}|,\;\;   |c_{k_{2}, n+1}-c_{k_{1},n+1}| \leq K, \quad \forall\; n\geq 1.
$$
 
Now we prove that the sequence $\{|\lambda_{n}|\}$ is also bounded above. 
Recall that when we chose $\gamma$, we assumed it did not go through $0$ and thus by the normalization, 
none of the $\gamma_{n+1}$ go through $0$ either.  Therefore, for each $n$ we can find a simply connected domain 
$D_{n+1} \subset \gamma_{n+1}$ that does not contain $0$. 
As in the previous section we compute 
$$
\eta=  \frac{1}{2\pi i} \oint_{\delta_{n}} \frac{1}{w} dw = \frac{1}{2\pi i} \int_{\gamma_{n+1}} \frac{E_{n}'(z)}{E_{n}(z)} dz
= \frac{1}{2\pi i} \int_{\gamma_{n+1}} \Big(\frac{p}{z} +\lambda_{n}\Big)dz
$$
so that as above
 $$
2 \pi  i \eta  = \int_{\gamma_{n+1}} \frac{p}{z} dz +\int_{\gamma_{n+1}} \lambda_{n} dz = \int_{\gamma_{n+1}} \frac{p}{z} dz + \lambda_{n} (c_{k_{2},n+1}-c_{k_{1},n+1}).
$$
Rewriting we have 
$$
\lambda_{n} (c_{k_{2},n+1}-c_{k_{1},n+1}) =2 \pi  i \eta - \int_{\gamma_{n+1}} \frac{p}{z} dz.
$$
This implies that
$$
\kappa |\lambda_{n}| \leq |\lambda_{n} (c_{k_{2},n+1}-c_{k_{1},n+1})| \leq  2 \pi  \eta + \Big|\int_{\gamma_{n+1}}\frac{p}{z} dz\Big| 
$$
so that if we can bound the integral on the right we will be done. 

Notice that  $\log z$ can be defined as an analytic function on the simply connected domain $D_{n+1}$ containing $\gamma_{n+1}$ that we chose above.  
We take $\log z =\log |z| +2\pi i \arg (z)$ as the principal branch, with $0\leq \arg (z)<2\pi$ . 
We then estimate 
$$
\Big|\int_{\gamma_{n+1}}\frac{p}{z} dz\Big|=| \log c_{k_{2},n+1} - \log c_{k_{1},n+1}|
$$
$$
\leq |\log |c_{k_{2},n+1}|-\log |c_{k_{1},n+1}| | + |\arg(c_{k_{2},n+1}) - \arg (c_{k_{1},n+1})| 
$$
$$
\leq 2\pi \eta + (\log K -\log \kappa) +4\pi .
$$
 Finally we have  
 $$
 |\lambda_{n}|  \leq \frac{2\pi \eta +(\log K-\log \kappa)  +4\pi}{\kappa}.
 $$ 
which proves that $\{ E_{n}(z)\}_{n=0}^{\infty}$ is contained in a compact subset in $\E_{p,1}$. 
This combined with Theorem~\ref{necc} and Theorem~\ref{main1} completes the proof of Theorem~\ref{main2} for $f\in \TE_{p,1}$.  

\section{Some Remarks}
One can formally define a Thurston obstruction for a post-singularly finite $(p,q)$-topological exponential map $f$ with $p\geq 0$ and $q\geq 1$.  Because such an $f$ is a branched covering of infinite degree, however, many arguments  in the proof of the Thurston theorem  that use the finiteness of the covering in an essential way  do not apply (see~\cite{DH,JZ,Ji}) .  Thus, how to define an analog of the Thurston obstruction to characterize a $(p,q)$-topological exponential map $f$ is not   clear to us.
We can, however, define an analog  of the {\em canonical} Thurston obstruction for a $(p,q)$-topological exponential map $f$ which depends on the hyperbolic lengths of curves (see to~\cite{Pi,CJ,Ji}).

Let $\sigma$ be the induced map on the Teichm\"uller space $T_{f}$. For any $\tau_{0}\in T_{f}$, and  for $n\geq 1$,
let $\tau_{n}=\sigma^{n}(\tau_{0})$. Let  $\gamma$ denote a simple closed non-peripheral curve
in ${\mathbb C}\setminus P_{f}$. Define
$$
\Gamma_{c}=\{ \gamma \;|\; \forall \tau_{0}\in T_{f}, \; l_{\tau_{n}}(\gamma) \to 0 \quad \hbox{as}\quad n \to \infty\}.
$$
We have

\begin{corollary}
If $\Gamma_{c}\not=\emptyset$, then $f$ has no bounded geometry and therefore, $f$ is not combinatorially equivalent to a $(p,q)$-exponential map.
\end{corollary}

The converse should also be true but we have no proof at this time.  The difficulty is that in the characterization of  post-critically finite    rational maps, many arguments in the proof of the converse statement use the finiteness of the covering in an essential way (see~\cite{Pi,CJ}).

A {\em Levy cycle} is a special Thurston obstruction for rational maps. It can be defined for a $(p,q)$-topological exponential map $f$ as follows.
A set
$$
\Gamma=\{ \gamma_{1}, \cdots, \gamma_{n}\}
$$
of simple closed non-peripheral curves in ${\mathbb C}\setminus P_{f}$ is called a Levy cycle
if for any $\gamma_{i}\in \Gamma$, there is a simple closed non-peripheral curve component
$\gamma'$ of $f^{-1}(\gamma_{i})$ such that $\gamma'$ is homotopic to $\gamma_{i-1}$ (we identify $\gamma_0$ with $\gamma_n$) rel $P_{f}$
and $f: \gamma'\to \gamma_{i}$ is a homeomorphism. Following a result in~\cite{HSS}, we   have

\begin{corollary}[\cite{HSS}]
Suppose $f$ is a $(0,1)$-topological exponential map with finite post-singular set. Then $f$ has no Levy cycle if and only if $f$ has bounded geometry.
\end{corollary}
We believe a similar result holds for all post-singularly finite maps in $\TE_{p,q}$ but we do not have a proof at this time.

\vspace*{20pt}
\noindent Tao Chen, Department of Mathematics, CUNY Graduate
School, New York, NY 10016. Email: chentaofdh@gmail.com

\vspace*{5pt}
\noindent Yunping Jiang, Department of Mathematics, Queens College of CUNY,
Flushing, NY 11367 and Department of Mathematics, CUNY Graduate
School, New York, NY 10016. Email: yunping.jiang@qc.cuny.edu

\vspace*{5pt}
\noindent Linda Keen, Department of Mathematics, Lehman College of CUNY,
Bronx, NY 10468
and Department of Mathematics, CUNY Graduate
School, New York, NY 10016
Email: LINDA.KEEN@lehman.cuny.edu

\end{document}